\documentclass[12pt]{amsart}
\usepackage[normalem]{ulem}
\usepackage{url}

\usepackage{verbatim}
\usepackage{amssymb}
\usepackage{upref}
\usepackage[all]{xy}
\usepackage[usenames,dvipsnames]{color}
\usepackage{url}
\usepackage[normalem]{ulem}

\allowdisplaybreaks

\newtheorem{thm}{Theorem}[section]
\newtheorem*{thm*}{Theorem}
\newtheorem{lem}[thm]{Lemma}

\newtheorem{prop}[thm]{Proposition}

\theoremstyle{definition}
\newtheorem{defn}[thm]{Definition}
\newtheorem{q}[thm]{Question}

\newtheorem*{notn*}{Notation}

\newtheorem*{hyp*}{Hypothesis}

\theoremstyle{remark}
\newtheorem{rem}[thm]{Remark}
\newtheorem*{rem*}{Remark}

\numberwithin{equation}{section}

\newcommand{\secref}[1]{Section~\textup{\ref{#1}}}

\newcommand{\thmref}[1]{Theorem~\textup{\ref{#1}}}

\newcommand{\propref}[1]{Proposition~\textup{\ref{#1}}}

\newcommand{\defnref}[1]{Definition~\textup{\ref{#1}}}
\newcommand{\remref}[1]{Remark~\textup{\ref{#1}}}

\newcommand{\midtext}[1]{\quad\text{#1}\quad}
\newcommand{\righttext}[1]{\quad\text{#1 }}

\newcommand{\N}{\mathbb N}

\newcommand{\T}{\mathbb T}

\newcommand{\CC}{\mathcal C}

\newcommand{\KK}{\mathcal K}

\newcommand{\GG}{\mathcal G}
\newcommand{\LL}{\mathcal L}

\newcommand{\OO}{\mathcal O}

\newcommand{\NN}{\mathcal N}

\renewcommand{\AA}{\mathcal A}

\renewcommand{\epsilon}{\varepsilon}

\DeclareMathOperator{\aut}{Aut}

\DeclareMathOperator{\ad}{Ad}

\DeclareMathOperator*{\spn}{span}
\DeclareMathOperator*{\clspn}{\overline{\spn}}

\newcommand{\id}{\text{\textup{id}}}
\newcommand{\case}{& \text{if }}

\newcommand{\<}{\langle}
\renewcommand{\>}{\rangle}

\newcommand{\iso}{\overset{\cong}{\longrightarrow}}

\renewcommand{\bar}{\overline}
\newcommand{\wilde}{\widetilde}

\begin{document}

\title{Ionescu's theorem for higher rank graphs}

\author[Kaliszewski]{S. Kaliszewski}
\address{School of Mathematical and Statistical Sciences
\\Arizona State University
\\Tempe, Arizona 85287}
\email{kaliszewski@asu.edu}

\author[Morgan]{Adam Morgan}
\address{School of Mathematical and Statistical Sciences
\\Arizona State University
\\Tempe, Arizona 85287}
\email{anmorgan@asu.edu}

\author[Quigg]{John Quigg}
\address{School of Mathematical and Statistical Sciences
\\Arizona State University
\\Tempe, Arizona 85287}
\email{quigg@asu.edu}

\subjclass[2000]{Primary  46L05}

\keywords{higher-rank graph $C^*$-algebra, Mauldin-Williams graph}

\date{\today}

\begin{abstract}
    We will define new constructions similar to the graph systems of correspondences
    described by Deaconu \emph{et al}. We will use these to prove a version
    of Ionescu's theorem for higher
    rank graphs. Afterwards we will examine the properties of these constructions
    further and make contact with Yeend's topological k-graphs and the tensor-groupoid-valued product systems of Fowler and Sims.
\end{abstract}
\maketitle

\section{Introduction}
\label{intro}

In \cite{ionescu}, Ionescu defines
a natural correspondence associated to any \emph{Mauldin-Williams} graph.
A Mauldin-Williams graph is a directed graph with a compact metric space
associated to each vertex and a contractive map associated to each edge
(a more rigorous definition is presented below). These graphs generalize
iterated function systems and have self-similar invariant sets. Ionescu
proved that the Cuntz-Pimsner algebra of the correspondence associated to
any Mauldin-Williams graph is isomorphic to the graph $C^*$-algebra of
the underlying graph.

Here we prove an analogue for higher rank graphs. Our arguments make extensive
use of the \emph{graph systems of correspondences} construction presented in \cite{system}
and 
(we hope)
provide an interesting application of 
their
ideas.  We also define some
other systems similar to those [defined in] of \cite{system} and briefly describe
how all of these systems fit into what Fowler and Sims refer to in \cite{fowlersims}
as \emph{product systems taking values in tensor groupoids}.

This paper is organized as follows. In 
\secref{prelim}
we will present a brief
overview of some of the preliminaries on $C^*$-correspondences, graph 
systems of correspondences, and topological k-graph algebras.  In 
\secref{systems}
we will define two systems that closely resemble $\Lambda$-systems of
correspondences which we will call $\Lambda$-systems of homomorphisms and
$\Lambda$-systems of maps. The $\Lambda$-system of maps will be a generalization
of the notion of a Mauldin-Williams
graph.  After defining some more terminology, we prove
some basic facts about these systems and how they relate to one another.
In 
\secref{MW}
we define a k-graph analog of Mauldin-Williams graphs and prove
our main theorem which generalizes Ionescu's main result from [Ion07]. In
\secref{topological}
we prove that the Cuntz-Pimsner algebra of the correspondence associated
to any $\Lambda$-system of maps can be realized as the graph algebra of
a certain topological k-graph. In 
\secref{tensor groupoid}
we briefly describe how all of
the various $\Lambda$-systems fit in to the framework described by Fowler
and Sims in \cite{fowlersims}. 
In 
\secref{reverse}
we will examine the question of
which $\Lambda$-systems of correspondences arise from the other types of
$\Lambda$-systems described here.
Finally, in \secref{no go} we show that, perhaps disappointingly, the higher-rank Mauldin-Williams graphs of \secref{MW} do not give rise to any new ``higher-rank fractals''.

\section{Preliminaries}
\label{prelim}

\subsection{Correspondences}
For $C^*$-algebras $A$ and $B$,
we usually want our
$A-B$ correspondences $X$
to be
\emph{nondegenerate}
in the sense that
$A\cdot X=X$, equivalently, the left-module homomorphism $\varphi_A:A\to \LL_B(X)=M(\KK_B(X))$ is nondegenerate.
If $\varphi:A\to M(B)$ is a nondegenerate homomorphism, the \emph{standard $A-B$ correspondence ${}_\varphi B$} is $B$ viewed as a Hilbert module in the usual way and equipped with the left $A$-module structure induced by $\varphi$.
The \emph{identity} $B$-correspondence is ${}_{\id}B$.

An
\emph{isomorphism} of an $A-B$ correspondence $X$ onto a $C-D$ correspondence $Y$
is a triple $(\theta_A,\theta,\theta_B)$,
where
\begin{itemize}
\item $\theta_A:A\to C$ and $\theta_B:B\to D$ are $C^*$-isomorphisms, and
\item $\theta:X\to Y$ is a linear bijection such that
\begin{align*}
&\<\theta(\xi),\theta(\eta)\>_D=\theta_B(\<\xi,\eta\>_B)
\righttext{for all}\xi,\eta\in X;\\
&\theta(a\cdot \xi\cdot b)=\theta_A(a)\cdot \theta(\xi)\cdot \theta_B(b)\righttext{for all}a\in A,\xi\in X,b\in B.
\end{align*}
\end{itemize}
$\theta_A$ and $\theta_B$ are the \emph{left-} and \emph{right-hand} \emph{coefficient isomorphisms}, respectively.
When both $X$ and $Y$ are $A-B$ correspondences,
we require, unless otherwise specified, that
the coefficient isomorphisms be the identity maps,
and we sometimes emphasize that we are making this assumption by saying that
$\theta:X\to Y$ is an \emph{$A-B$ correspondence isomorphism}.

Recall from \cite[Proposition~2.3]{taco} that for two nondegenerate homomorphisms $\varphi,\psi:A\to M(B)$,
the standard $A-B$ correspondences ${}_\varphi B$ and ${}_\psi B$ are isomorphic
if and only if there is a unitary multiplier $u\in M(B)$ such that $\psi=\ad u\circ\varphi$
(the special case of imprimitivity bimodules is essentially
\cite[Proposition~3.1]{bgr}).
In particular, if $B$ is commutative then ${}_\varphi B\cong {}_\psi B$ if and only if $\varphi=\psi$.

\subsection{$\Lambda$-systems}
Throughout, $\Lambda$ will be a row-finite $k$-graph with no sources, so that the associated Cuntz-Krieger relations take the most elementary form.
In \cite{system}, Deaconu, Kumjian, Pask, and Sims introduced
\emph{$\Lambda$-systems}
of correspondences:
we have a Banach bundle $X\to \Lambda$ with fibres $\{X_\lambda\}_{\lambda\in\Lambda}$ such that
\begin{enumerate}
\item for each $v\in \Lambda^0$, $X_v$ is a $C^*$-algebra;

\item for each $\lambda\in u\Lambda v$, $X_\lambda$ is an $X_u-X_v$ correspondence;

\item there is a partially-defined associative multiplication on $X=\bigsqcup_{\lambda\in \Lambda}X_\lambda$ that is compatible with the multiplication in $\Lambda$ via the bundle projection $X\to \Lambda$;

\item whenever $\lambda,\mu\notin \Lambda^0$ and $s(\lambda)=r(\mu)$, $x\otimes y\mapsto xy:X_\lambda\otimes_{A_{s(\lambda)}} X_\mu\to X_{\lambda\mu}$ is an isomorphism of $A_{r(\lambda)}-A_{s(\mu)}$ correspondences;

\item the left and right module multiplications of the correspondences coincide with the multiplication from the $\Lambda$-system.
\end{enumerate}
For a $\Lambda$-system $X$ of correspondences, we will write
\[
\varphi_\lambda:X_u\to \LL(X_\lambda)\righttext{for}\lambda\in u\Lambda
\]
for the left-module structure map.
Note that the multiplication in $X$
induces
$X_u-X_v$ correspondence isomorphisms
$X_\lambda\otimes X_v\cong X_\lambda$
for all $\lambda\in u\Lambda v$,
but only induces isomorphisms
$X_u\otimes X_\lambda\cong X_\lambda$
if every correspondence $X_\lambda$ is nondegenerate.

Given $\Lambda$-systems $X$ and $Y$ of correspondences,
a map $\theta:X\to Y$ is
a \emph{$\Lambda$-system isomorphism}
if
\begin{enumerate}
\item for all $\lambda\in u\Lambda v$,
\[
\theta_\lambda:=\theta|_{X_\lambda}:X_\lambda\to Y_\lambda
\]
is an isomorphism of correspondences with coefficient isomorphisms $\theta_u,\theta_v$;

\item for all $\lambda\in u\Lambda v,\mu\in v\Lambda w$,
\[
\theta_\lambda(\xi)\theta_\mu(\eta)=\theta_{\lambda\mu}(\xi\eta)\righttext{for all}\xi\in X_\lambda,\eta\in X_\mu.
\]
\end{enumerate}
Since the multiplication in the $\Lambda$-system induces the left and right module multiplications for the correspondences, in the above we can relax (1) to
\begin{enumerate}
\item[$(1)'$]
for all $\lambda\in \Lambda v$, $\theta_\lambda:X_\lambda\to Y_\lambda$ is a linear bijection satisfying
\begin{equation*}
\<\theta_\lambda(\xi),\theta_\lambda(\eta)\>_{Y_v}=\theta_v(\<\xi,\eta\>_{X_v})\righttext{for all}\xi,\eta\in X_\lambda,
\end{equation*}
\end{enumerate}
because (2) takes care of the coefficient maps.
We emphasize that we're requiring that,
for each $v\in\Lambda^0$,
$\theta_v$ be the right-hand coefficient isomorphism
for every correspondence isomorphism $\theta_\lambda$ with $s(\lambda)=v$,
and also the left-hand coefficient isomorphism
for every correspondence isomorphism $\theta_\lambda$ with $r(\lambda)=v$.
Thus, if $X$ and $Y$ are isomorphic $\Lambda$-systems of correspondences, then 
without loss of generality we may assume (if we wish) that $X_v=Y_v$
and $\theta_v=\id_{X_v}$ for every vertex $v$,
so that
$\theta_\lambda:X_\lambda\to Y_\lambda$ 
is an $X_u-X_v$ correspondence isomorphism 
whenever $\lambda\in u\Lambda v$.

\subsection{Topological $k$-graphs}
Recall \cite{yeend:topk-graph}
that a \emph{topological $k$-graph}
is a $k$-graph $\Gamma$ equipped with a locally compact Hausdorff topology
making the multiplication continuous and open,
the range map continuous,
the source map a local homeomorphism,
and the degree functor $d:\Gamma\to \N^k$ continuous.
Carlsen, Larsen, Sims, and Vittadello show in \cite[Proposition~5.9]{carlsen} that every topological $k$-graph $\Gamma$ gives rise to a $\N^k$-system $Z$ of correspondences over $A:=C_0(\Gamma^0)$
as follows:
For each $n\in\N^k$ let $Z_n$ be the $A$-correspondence
associated to the topological graph
$(\Gamma^0,\Gamma^n,s|_{\Gamma^n},r|_{\Gamma^n})$,
so that $Z_n$ is the completion of the pre-correspondence $C_c(\Gamma^n)$ with operations
\begin{align*}
&(f\cdot \xi\cdot g)(\alpha)=f(r(\alpha))\xi(\alpha)g(s(\alpha))\\
&\<\xi,\eta\>_A(v)=\sum_{\alpha\in \Gamma^nv}\bar{\xi(\alpha)}\eta(\alpha),
\end{align*}
for $\xi,\eta\in C_c(\Gamma^n)$, $f,g\in A$.
Then for $\xi\in C_c(\Gamma^n),\eta\in C_c(\Gamma^m)$ define $\xi\eta\in C_c(\Gamma^{n+m})$ by
\[
(\xi\eta)(\alpha\beta)=\xi(\alpha)\eta(\beta)\righttext{for}\alpha\in \Gamma^n,\beta\in \Gamma^m,s(\alpha)=r(\beta).
\]
In \cite{yeend:topk-graph}, Yeend defined $C^*(\Gamma)$ using a groupoid model,
but \cite[Theorem~5.20]{carlsen} shows that $C^*(\Gamma)\cong \NN\OO_Z$,
where $\NN\OO_Z$ is the Cuntz-Nica-Pimsner algebra of the product system $Z$.
The topological $k$-graphs we encounter in this paper will be nice enough that $\NN\OO_Z$ will coincide with the Cuntz-Pimsner algebra $\OO_Z$.

\section{Other $\Lambda$-systems}\label{systems}

We introduce a few constructions that are similar to $\Lambda$-systems of correspondences:
\begin{defn}
\begin{enumerate}
\item
A \emph{$\Lambda$-system of homomorphisms} is a 
pair $(\AA,\varphi)$, where 
$\AA\to \Lambda^0$ is a $C^*$-bundle
and for each $\lambda\in u\Lambda v$ we have a
nondegenerate homomorphism $\varphi_\lambda:A_u\to M(A_v)$,
such that
\begin{align*}
&\varphi_{\mu}\circ\varphi_{\lambda}=\varphi_{\lambda\mu}\righttext{if}s(\lambda)=r(\mu)\\
&\varphi_v=\id_{A_v}\righttext{for}v\in \Lambda^0,
\end{align*}
where we have canonically extended $\varphi_\mu$ to $M(A_v)$.

\item
A \emph{$\Lambda$-system of maps} is a 
pair $(T,\sigma)$, where 
$T\to \Lambda^0$ is a bundle 
of locally compact Hausdorff spaces 
and for each $\lambda\in u\Lambda v$ we have a
continuous map $\sigma_\lambda:T_v\to T_u$,
such that
\begin{align*}
&\sigma_{\lambda}\circ\sigma_{\mu}=\sigma_{\lambda\mu}\righttext{if}s(\lambda)=r(\mu)\\
&\sigma_v=\id_{T_v}\righttext{for}v\in \Lambda^0.
\end{align*}
\end{enumerate}
\end{defn}

\begin{rem}\label{processes}
\begin{enumerate}
\item
Note that we need to impose the nondegeneracy condition on the homomorphisms $\varphi_\lambda$ so that composition is defined.

\item
Thus, a $\Lambda$-system of homomorphisms is essentially a contravariant functor from $\Lambda$ to the category of $C^*$-algebras and nondegenerate homomorphisms into multiplier algebras,
and a $\Lambda$-system of maps is essentially a (covariant) functor from $\Lambda$ to the category of locally compact Hausdorff spaces and continuous maps.

\item
Every $\Lambda$-system $(T,\sigma)$ of maps gives rise to a $\Lambda$-system $(\AA,\sigma^*)$ of homomorphisms,
with
\begin{align*}
&A_v=C_0(T_v)\righttext{for}v\in \Lambda^0\\
&\sigma_\lambda^*(f)=f\circ\sigma_\lambda\righttext{for}\lambda\in \Lambda,f\in A_{r(\lambda)}.
\end{align*}

\item
Every $\Lambda$-system $(\AA,\varphi)$ of homomorphisms gives rise to a $\Lambda$-system of correspondences:
for $\lambda\in u\Lambda v$ let $X_\lambda$ be the 
standard
$A_u-A_v$ correspondence ${}_{\varphi_\lambda}A_v$.
\end{enumerate}
\end{rem}

\begin{defn}
We call a $\Lambda$-system of maps $(T,\sigma)$
\begin{enumerate}
\item
\emph{proper} if each map $\sigma_\lambda:T_{s(\lambda)}\to T_{r(\lambda)}$ is proper (in the usual sense that inverse images of compact sets are compact), and

\item
\emph{dense} if each map $\sigma_\lambda:T_{s(\lambda)}\to T_{r(\lambda)}$ has dense range.
\end{enumerate}
\end{defn}

\begin{defn}
We call a 
$C^*$-homomorphism $\varphi:A\to M(B)$ \emph{proper} if it maps into $B$
(and we will also denote it by $\varphi:A\to B$).
\end{defn}

\begin{rem}
A nondegenerate homomorphism 
$\varphi:A\to M(B)$ is 
proper in the above sense if and only if $\varphi$ takes one (hence every) bounded approximate identity for $A$ to an approximate identity for $B$.
Also, if 
$\sigma:Y\to X$ is a continuous map,
then $\sigma^*:C_0(X)\to M(C_0(Y))$
is automatically nondegenerate, and
is proper if and only if $\sigma$ is proper.
\end{rem}

\begin{defn}
Let $X$ an $A-B$ correspondence, with left module map $\varphi_A:A\to \LL(X)=M(\KK(X))$.
We call $X$
\emph{proper}, \emph{nondegenerate}, or  \emph{faithful} if
$\varphi_A$ has the corresponding property.
\end{defn}

\begin{defn}
We call a $\Lambda$-system $(\AA,\varphi)$ of homomorphisms
\emph{proper} or \emph{faithful}
if each homomorphism $\varphi_\lambda$ has the corresponding property.
\end{defn}

\begin{defn}
We call a $\Lambda$-system $X$ of correspondences
\emph{proper}, \emph{nondegenerate}, \emph{full}, or \emph{faithful}
if each correspondence $X_\lambda$ has the corresponding property.
\end{defn}

\cite{system} calls a $\Lambda$-system $X$ of correspondences
\emph{regular}
if it is proper, nondegenerate, full, and faithful.
However, we believe that the fidelity is too much to ask, both for aesthetic and practical reasons.

Let $X$ be a $\Lambda$-system of correspondences,
and let $A=\bigoplus_{v\in \Lambda^0}X_v$ be the $c_0$-direct sum of $C^*$-algebras.
Then each $X_\lambda$ may be regarded as an $A$-correspondence.
For each $n\in\N^k$, \cite{system} defines an $A$-correspondence $Y_n$ by
\[
Y_n=\bigoplus_{\lambda\in \Lambda^n}X_\lambda,
\]
and \cite[Proposition~3.17]{system} shows that $Y=Y_X=\bigsqcup_{n\in\N^k}Y_n$ is an $\N^k$-system (i.e., a product system over $\N^k$) of $A$-correspondences,
and moreover if $X$ is regular then so is $Y$.
We will identify $X_\lambda$ with its canonical image in $Y_{d(\lambda)}$,
i.e., we will blur the distinction between the external and internal direct sums of the $A$-correspondences $\{X_\lambda:\lambda\in \Lambda^n\}$.

\begin{defn}
We call a $\Lambda$-system $(T,\sigma)$ of maps
\begin{enumerate}
\item
\emph{$k$-dense}
if for all $v\in \Lambda^0$ and $n\in\N^k$,
\[
T_v=\overline{\bigcup_{\lambda\in v\Lambda^n}\sigma_\lambda(T_{s(\lambda)})},
\]
and

\item
\emph{$k$-regular}
if it is proper and $k$-dense.
\end{enumerate}
\end{defn}

Here is a minor strengthening of $k$-density that we will find useful later:

\begin{defn}
A $\Lambda$-system of maps $(T,\sigma)$ is \emph{$k$-surjective} if
\[
T_v=\bigcup_{\lambda\in v\Lambda^n}\sigma_\lambda\bigl(T_{s(\lambda)}\bigr)\righttext{for all}v\in\Lambda^0,n\in\N^k.
\]
\end{defn}

\begin{defn}
We call a $\Lambda$-system $(\AA,\varphi)$ of homomorphisms
\begin{enumerate}
\item
\emph{$k$-faithful}
if for all $v\in \Lambda^0$ and $n\in\N^k$,
\[
\bigcap_{\lambda\in v\Lambda^n}\ker\varphi_\lambda=\{0\},
\]
and

\item
\emph{$k$-regular}
if it is proper and $k$-faithful.
\end{enumerate}
\end{defn}

\begin{defn}\label{k-faithful}
We call a $\Lambda$-system $X$ of correspondences
\begin{enumerate}
\item
\emph{$k$-faithful}
if the associated $\N^k$-system $Y_X$
is faithful,
and

\item
\emph{$k$-regular}
if it is proper, nondegenerate, full, and $k$-faithful.
\end{enumerate}
\end{defn}

\begin{rem}
\begin{enumerate}
\item
If $(T,\sigma)$ is a $\Lambda$-system of maps,
then the associated $\Lambda$-system $(\AA,\sigma^*)$ of homomorphisms is:
\begin{itemize}
\item
proper if and only if $(T,\sigma)$ is,
and

\item
faithful if and only if $(T,\sigma)$ is dense.
\end{itemize}

\item
If $(\AA,\varphi)$ is a $\Lambda$-system of homomorphisms,
then the associated $\Lambda$-system $X$ of correspondences is:
\begin{itemize}
\item
automatically nondegenerate and full,
and

\item
proper or faithful
if and only if $(\AA,\varphi)$ has the corresponding property.
\end{itemize}

\item
We have organized our definitions so that a $\Lambda$-system $X$ of correspondences is $k$-regular if and only if the associated $\N^k$-system $Y_X$ is regular.
\end{enumerate}
\end{rem}

We will need the following variation on the above:

\begin{lem}
\begin{enumerate}
\item
If $(T,\sigma)$ is a $\Lambda$-system of maps,
then the associated $\Lambda$-system $(\AA,\sigma^*)$ of homomorphisms is
$k$-faithful if and only if $(T,\sigma)$ is $k$-dense,
and consequently is $k$-regular if and only if $(T,\sigma)$ is.

\item
If $(\AA,\varphi)$ is a $\Lambda$-system of homomorphisms,
then the associated $\Lambda$-system $X$ of correspondences is
$k$-faithful if and only if $(\AA,\varphi)$ is,
and consequently is $k$-regular if and only if $(X,\varphi)$ is.
\end{enumerate}
\end{lem}

\begin{proof}
(1).
This is routine, but we present the details for completeness.
First assume that $(T,\sigma)$ is not $k$-dense.
We will show that $(\AA,\sigma^*)$ is not $k$-faithful.
We can choose $v\in \Lambda^0$ and $n\in\N^k$ such that $\bigcup_{\lambda\in v\Lambda^n}\sigma_\lambda(T_{s(\lambda)})$ is not dense in $T_v$.
We will show that $\bigcap_{\lambda\in v\Lambda^n}\ker\sigma^*_\lambda\ne \{0\}$.
We can choose a nonzero $f\in C_0(T_v)$ that vanishes on $\bigcup_{\lambda\in v\Lambda^n}\sigma_\lambda(T_{s(\lambda)})$.
Then for all $\lambda\in v\Lambda^n$ and all $g\in C_0(T_{s(\lambda)})$,
\[
\sigma_\lambda^*(f)g=(f\circ\sigma_\lambda)g=0.
\]
Thus $f\in \bigcap_{\lambda\in v\Lambda^n}\ker\sigma^*_\lambda$.

Conversely, assume that $(\AA,\sigma^*)$ is not $k$-faithful.
We will show that $(T,\sigma)$ is not $k$-dense.
We can choose $v\in \Lambda^0$ and $n\in\N^k$ such that $\bigcap_{\lambda\in v\Lambda^n}\ker\sigma_\lambda^*\ne \{0\}$.
We will show that $\bigcup_{\lambda\in v\Lambda^n}\sigma_\lambda(T_{s(\lambda)})$ is not dense in $T_v$.
Choose a nonzero $f\in \bigcap_{\lambda\in v\Lambda^n}\ker \sigma_\lambda^*$.
Then choose a nonempty open set $U\subset T_v$ such that $f(t)\ne 0$ for all $t\in U$.
We will show that
\[
U\cap \bigcup_{\lambda\in v\Lambda^n}\sigma_\lambda(T_{s(\lambda)})=\varnothing.
\]
Let $t\in \bigcup_{\lambda\in v\Lambda^n}\sigma_\lambda(T_{s(\lambda)})$,
and choose $\lambda\in v\Lambda^n$ and $s\in T_{s(\lambda)}$ such that $t=\sigma_\lambda(s)$.
Then choose $g\in C_0(T_{s(\lambda)})$ such that $g(s)=1$.
Since $f\in \ker\sigma_\lambda^*$,
\[
0=\bigl(\sigma_\lambda^*(f)g\bigr)(s)=f(\sigma_\lambda(s))g(s)=f(t),
\]
so $t\notin U$.

(2).
First assume that $(\AA,\varphi)$ is not $k$-faithful.
We will show that $X$ is not $k$-faithful.
We can choose $v\in \Lambda^0$ and $n\in\N^k$ such that $\bigcap_{\lambda\in v\Lambda^n}\ker\varphi_\lambda\ne \{0\}$.
We will show that the $A$-correspondence $Y_n$ is not faithful.
Choose a nonzero $a\in A_v$ 
such that $\varphi_\lambda(a)=0$ for all $\lambda\in v\Lambda^n$.
Let
\[
y=(x_\lambda)\in Y_n=\bigoplus_{\lambda\in \Lambda^n}X_\lambda.
\]
Then
$a\cdot y$ is the $\Lambda^n$-tuple $(a\cdot x_\lambda)$,
where for $\lambda\in \Lambda^n$ we have
\[
a\cdot x_\lambda=\begin{cases}\varphi_\lambda(a)x_\lambda\case r(\lambda)=v\\0\case r(\lambda)\ne v.\end{cases}
\]
Since $\varphi_\lambda(a)x_\lambda=0$ for all $\lambda\in v\Lambda^n,x_\lambda\in X_\lambda=A_{s(\lambda)}$, we have $a\cdot y=0$, and we have shown that 
$Y_n$ is not faithful.

Conversely, assume that
$X$ is not $k$-faithful.
We will show that $(\AA,\varphi)$ is not $k$-faithful.
We can choose $n\in\N^k$ such that the $A$-correspondence $Y_n$ is not faithful,
so we can find a nonzero $a\in A$ such that $a\cdot y=0$ for all $y\in Y_n$.
Then $a=(a_v)$ is a $\Lambda^0$-tuple with $a_v\in A_v$ for each $v$,
so we can choose $v\in \Lambda^0$ such that $a_v\ne 0$.
We will show that $a_v\in \bigcap_{\lambda\in v\Lambda^n}\ker \varphi_\lambda$.
Let $\lambda\in v\Lambda^0$ and $b\in A_{s(\lambda)}$.
Define a $v\Lambda^n$-tuple $(x_\mu)\in Y_n$ by
\[
x_\mu=\begin{cases}b\case \mu=\lambda\\ 0\case \mu\ne \lambda.\end{cases}
\]
Then
\[
\varphi_\lambda(a_v)b=\bigl(a_v\cdot (x_\lambda)\bigr)_\lambda=0.
\qedhere
\]
\end{proof}

\begin{rem}
The argument in the last paragraph of the above proof is a routine adaptation of that used in \cite[Proposition~3.1.7]{system}.
\end{rem}

Our motivation for introducing the properties of $k$-density and $k$-fidelity is that the Mauldin-Williams graphs considered by Ionescu ---
where we have a 1-graph $\Lambda$
whose 1-skeleton $E$ is finite,
a $\Lambda$-system $(T,\sigma)$ of maps
in which each space $T_v$ is a compact metric space
and each map $\sigma_\lambda$ is a (strict) contraction ---
are typically 1-dense in the above sense rather than dense.
More precisely, a Mauldin-Williams graph $(T,\sigma)$ is 
dense (in our terminology) if and only if
every map $\sigma_e$ (for $e\in E^1$) is surjective,
which is usually not the case,
and
1-dense if and only if
for all $v\in E^0$ we have
\[
\bigcup_{e\in vE^1}\sigma_e(T_{s(e)})=T(v),
\]
which is always the case
(after replacing the original spaces by an ``invariant list'').
Thus, since we want to consider a version of Ionescu's theorem for $k$-graphs,
we must allow the weakened property of $k$-fidelity (of \defnref{k-faithful}) rather than insisting upon fidelity.

\cite[Definition~3.2.1]{system} defines a \emph{representation} of a $\Lambda$-system $X$ in a $C^*$-algebra $B$ as a map $\rho:X\to B$ such that
\begin{enumerate}
\item for each $v\in \Lambda^0$, $\rho_v:X_v\to B$ is a $C^*$-homomorphism;

\item whenever $\xi\in X_\lambda,\eta\in X_\mu$,
\[
\rho_\lambda(\xi)\rho_\mu(\eta)=\begin{cases}
\rho_{\lambda\mu}(\xi\eta)\case s(\lambda)=r(\mu)\\
0\case \text{otherwise};
\end{cases}
\]

\item whenever $\xi\in X_\lambda,\eta\in X_\mu$,
if $d(\lambda)=d(\mu)$ then
\[
\rho_\lambda(\xi)^*\rho_\mu(\eta)=\begin{cases}
\rho_{s(\lambda)}(\<\xi,\eta\>_{X_{s(\lambda)}})\case \lambda=\mu\\
0\case \text{otherwise},
\end{cases}
\]
\end{enumerate}
and when $X$ is regular \cite{system} defines a representation $\rho$ to be \emph{Cuntz-Pimsner covariant} if
for all $v\in\Lambda^0$, $n\in\N^k$, and $a\in X_v$,
\begin{enumerate}
\setcounter{enumi}{3}
\item
\[
\rho_v(a)=\sum_{\lambda\in v\Lambda^n}\rho^{(\lambda)}(\varphi_\lambda(a)),
\]
\end{enumerate}
where $\rho^{(\lambda)}=\rho_\lambda^{(1)}:\KK(X_\lambda)\to B$ is the associated homomorphism.
Then \cite{system} defines a representation $\rho$ to be \emph{universal} if for every representation $\tau:X\to C$ there is a unique $C^*$-homomorphism
$\Phi=\Phi_\tau:B\to C$ such that $\Phi\circ\rho_\lambda=\tau_\lambda$ for all $\lambda\in\Lambda$,
and a Cuntz-Pimsner covariant representation to be \emph{universal} if it satisfies the above universality property for all Cuntz-Pimsner covariant representations.
Then \cite{system} points out that, by the nondegeneracy that is included in the regularity assumption,
(1)--(3) above can be replaced by the following set of conditions:
each $\rho_\lambda$ is a correspondence representation in $B$,
$\rho$ is multiplicative whenever this makes sense,
and $\rho_u$ and $\rho_v$ have orthogonal images for all $u\ne v\in\Lambda^0$.

For the $\N^k$-system $Y=Y_X$ associated to a regular $\Lambda$-system $X$,
\cite[Proposition~3.2.3]{system} shows that there is a bijection between the representations $\rho:X\to B$ and the representations $\psi:Y\to B$ such that
\[
\psi\circ\iota_\lambda=\rho_\lambda\righttext{for all}\lambda\in\Lambda.
\]
\emph{However, it is crucial for our results to note that the proof of \cite[Proposition~3.2.3]{system} only requires nondegeneracy of $Y$, not of $X$.}

\cite[Proposition~3.2.5]{system} shows that if $X$ is regular then a representation $\rho:X\to B$ is Cuntz-Pimsner covariant if and only if the associated representation $\psi:Y\to B$ is.
We turn this result into a \emph{definition}:

\begin{defn}\label{cov}
Let $X$ be a $k$-regular $\Lambda$-system of correspondences,
with associated $\N^k$-system $Y$,
and let $\rho:X\to B$ be a representation of $X$,
with associated representation $\psi:Y\to B$.
We define $\rho$ to be \emph{Cuntz-Pimsner covariant} if 
$\psi$ is,
in other words
\[
\sum_{\lambda\in v\Lambda^n}\rho^{(\lambda)}\circ\varphi_\lambda=\rho_v\righttext{for all}v\in\Lambda^0.
\]
\end{defn}

\begin{rem}\label{CP}
To reiterate, the only difference between \defnref{cov} and the definition of Cuntz-Pimsner covariance given in \cite[Definition~3.2.1]{system} is that in the latter the $\Lambda$-system $X$ is required to be regular, while we only require $k$-regularity.
In any event, \cite[Definition~3.2.7]{system} defines the $C^*$-algebra of a regular $\Lambda$-system $X$ to be the Cuntz-Pimsner algebra $\OO_Y$,
and in \cite[Corollary~3.2.6]{system} they notice that the representation $\rho^{j_Y}:X\to \OO_Y$ is a universal Cuntz-Pimsner covariant representation, where $j_Y:Y\to \OO_Y$ is the universal Cuntz-Pimsner covariant representation.

We emphasize that, even though we only require the $\Lambda$-system $X$ to be $k$-regular, the theory of \cite{system} carries over with no problems, as we've indicated.
\cite{system} uses the notation $C^*(A,X,\chi)$ for the $C^*$-algebra of $X$, but we'll write it as $\OO_X$.
If $\rho:X\to B$ is any Cuntz-Pimsner covariant representation, we'll write $\Phi_\rho:\OO_X\to B$ for the homomorphism whose existence is guaranteed by universality; \cite{system} would write it as $\Phi_{\rho,\pi}$, because they write $\pi$ for the restriction of $\rho$ to the $C^*$-bundle
$X|_{\Lambda^0}$ (and they write $A$ for this $C^*$-bundle, as well as for the section algebra $\bigoplus_{v\in\Lambda^0}X_v$ --- we reserve the name $A$ for this latter $C^*$-algebra).
\end{rem}

Note that since we assume that $\Lambda$ is row-finite with no sources,
the infinite-path space
$\Lambda^\infty$
is locally compact Hausdorff,
and is the disjoint union of the compact open subsets $\{v\Lambda^\infty\}_{v\in\Lambda^0}$.
We get a $\Lambda$-system of maps $(\Lambda^\infty,\tau)$,
where for $\lambda\in u\Lambda v$ the continuous map
\[
\tau_\lambda:v\Lambda^\infty\to u\Lambda^\infty
\]
is defined by $\tau_\lambda(x)=\lambda x$.
Moreover, this $\Lambda$-system is $k$-regular.
This system has the following properties:
if $\lambda\in u\Lambda v$ then $\tau_\lambda$ is a homeomorphism of $v\Lambda^\infty$ onto the compact open set
\[
\lambda\Lambda^\infty\subset u\Lambda^\infty,
\]
and consequently $\tau_\lambda^*$ is a surjection of $C(u\Lambda^\infty)$ onto $C(v\Lambda^\infty)$.

\begin{lem}
For each $u\in \Lambda^0$
and $\lambda\in u\Lambda$
let $p_\lambda=s_\lambda s_\lambda^*$,
the set
\[
D_u=\clspn\{p_\lambda:\lambda\in u\Lambda\}
\]
is a unital commutative $C^*$-subalgebra of $C^*(\Lambda)$, with unit $p_u$,
and the subalgebras $\{D_u\}_{u\in \Lambda^0}$ are pairwise orthogonal.
Moreover, if $D$ denotes the 
commutative $C^*$-subalgebra 
$\bigoplus_{u\in \Lambda^0}D_u$
of $C^*(\Lambda)$,
then there 
is an isomorphism $\theta:C_0(\Lambda^\infty)\to D$
that takes the characteristic function of $\lambda\Lambda^\infty=\{\lambda x:s(\lambda)=r(x)\}$ to $p_\lambda$
and
$C(u\Lambda^\infty)$ to $D_u$.
Finally, the diagram
\[
\xymatrix{
C(u\Lambda^\infty) \ar[r]^-{\tau_\lambda^*} \ar[d]_\theta
&C(v\Lambda^\infty) \ar[d]^\theta
\\
D_u \ar[r]_-{\ad s_\lambda^*}
&D_v
}
\]
commutes.
\end{lem}

\begin{proof}
This is probably folklore, at least for directed graphs, and in any case is standard:
truncation gives an inverse system $\{\Lambda^n,\tau_{m,n}\}$ of surjections among nonempty finite sets\footnote{with $\tau_{m,n}(\lambda)=\lambda(0,n)$ for $\lambda\in\Lambda^m$ and $n\le m$}, whose inverse limit is $\Lambda^\infty$,
and for each $n$ the commutative $C^*$-subalgebra $D^n:=\clspn\{p_\lambda:\lambda\in\Lambda^n\}$ of $C^*(\Lambda)$ has spectrum $\Lambda^n$, so the inductive limit $D=\clspn\{D^n:n\in\N^k\}$ has spectrum $\Lambda^\infty$.
\end{proof}

\begin{defn}
Let $(T,\sigma)$ be a $\Lambda$-system of maps.
A continuous map
$\Phi:\Lambda^\infty\to T$ 
is \emph{intertwining}
if
\[
\Phi\circ\tau_\lambda=\sigma_\lambda\circ\Phi\righttext{for all}\lambda\in\Lambda.
\]
We say $(T,\sigma)$
is \emph{self-similar} if there is a surjective intertwining map $\Phi:\Lambda^\infty\to T$.
\end{defn}

\begin{prop}\label{self}
Every self-similar $\Lambda$-system of maps $(T,\sigma)$ is $k$-surjective, and each space $T_v$ is compact.
\end{prop}

\begin{proof}
First, $T_v$ is compact because the intertwining property and surjectivity of $\Phi$ imply that $T_v=\Phi(v\Lambda^\infty)$, which is a continuous image of the compact set $v\Lambda^\infty$.
For the $k$-surjectivity, 
if $v\in\Lambda^0$ and $n\in\N^k$ then
\begin{align*}
T_v
&=\Phi(v\Lambda^\infty)
\\&=\Phi\left(\bigcup_{\lambda\in v\Lambda^n}\lambda\Lambda^\infty\right)
\\&=\bigcup_{\lambda\in v\Lambda^n}\Phi(\lambda\Lambda^\infty)
\\&=\bigcup_{\lambda\in v\Lambda^n}\sigma_\lambda\bigl(\Phi(s(\lambda)\Lambda^\infty)\bigr)
\\&=\bigcup_{\lambda\in v\Lambda^n}\sigma_\lambda(T_{s(\lambda)}).
\qedhere
\end{align*}
\end{proof}

\begin{defn}
Let $(T,\sigma)$ be a $\Lambda$-system of maps,
and let $S\subset T$ be locally compact.
For each $v\in\Lambda^0$ let $S_v=S\cap T_v$.
Suppose that 
\[
\sigma_\lambda(S_v)\subset S_u\righttext{whenever}\lambda\in u\Lambda v.
\]
Then $(S,\sigma|_S)$ is a \emph{$\Lambda$-subsystem} of $(T,\sigma)$,
where
\[
(\sigma|_S)_\lambda=\sigma_\lambda|_{S_v}\righttext{for all}\lambda\in \Lambda_v.
\]
\end{defn}
Note that our terminology makes sense: every $\Lambda$-subsystem is in fact a $\Lambda$-system.

\begin{prop}\label{T' self}
Let $(T,\sigma)$ be a $\Lambda$-system of maps,
and let $\Phi:\Lambda^\infty\to T$ be an intertwining map.
Put
\begin{align*}
&T'_v=\Phi(v\Lambda^\infty)\righttext{for each}v\in\Lambda^0\\
&T'=\bigcup_{v\in\Lambda^0}T'_v.
\end{align*}
Then
$(T',\sigma|_{T'})$ is a self-similar $k$-surjective $\Lambda$-subsystem of $(T,\sigma)$, and each $T'_v$ is compact.
\end{prop}

\begin{proof}
First of all, each $T'_v$ is compact 
since $v\Lambda^\infty$ is compact and $T_v$ is Hausdorff.
Thus
$T'$
is locally compact, since the sets $T_v$ are pairwise disjoint and open.
For each $\lambda\in u\Lambda v$ we have
\begin{align*}
\sigma_\lambda(T'_v)
&=\sigma_\lambda(\Phi(v\Lambda^\infty))
\\&=\Phi(\lambda\Lambda^\infty)
\\&\subset \Phi(u\Lambda^\infty)
\\&=T'_u,
\end{align*}
so $(T',\sigma|_{T'})$ is a $\Lambda$-subsystem of $(T,\sigma)$.
It is self-similar because $\Phi$ is intertwining and maps $\Lambda^\infty$ onto $T'$ by construction.
Then by \propref{self} $(T',\sigma|_{T'})$ is $k$-surjective.
\end{proof}

%MAIN RESULT:

\begin{thm}\label{main}
Let $(T,\sigma)$ be a self-similar $k$-regular $\Lambda$-system of maps,
and let $X$ be the associated $\Lambda$-system of correspondences.
Then
\[
\OO_X\cong C^*(\Lambda).
\]
\end{thm}

\begin{proof}
Our strategy will be to find a Cuntz-Pimsner covariant representation $\rho:X\to C^*(\Lambda)$ whose image contains the generators, and then apply the Gauge-Invariant Uniqueness Theorem.
Recall that for $\lambda\in u\Lambda v$,
$X_\lambda$ is the standard $C_0(T_u)-C_0(T_v)$ correspondence ${}_{\sigma_\lambda^*}C_0(T_v)$.
Define $\rho_\lambda:X_\lambda\to C^*(\Lambda)$ by
\[
\rho_\lambda(f)=s_\lambda \theta\circ\Phi^*(f)\righttext{for}f\in C_0(T_v).
\]
Then $\rho_\lambda$ is linear, and for $f,g\in C_0(T_v)$ we have
\begin{align*}
\rho_\lambda(f)^*\rho_\lambda(g)
&=\theta(\Phi^*(\bar f))s_\lambda^*s_\lambda\theta(\Phi^*(g))
\\&=\theta(\Phi^*(\bar f))p_v\theta(\Phi^*(g))
\\&=\theta(\Phi^*(\bar fg))
\\&=p_v\bigl(\<f,g\>_{C(T_v)}\bigr).
\end{align*}
For $\lambda\in \Lambda v$, $\mu\in v\Lambda w$,
$f\in C_0(T_v)$, and $h\in C_0(T_w)$ we have
\begin{align*}
\rho_\lambda(f)\rho_\mu(h)
&=s_\lambda \theta(\Phi^*(f))s_\mu \theta(\Phi^*(h))
\\&=s_\lambda \theta(\Phi^*(f))p_\mu s_\mu \theta(\Phi^*(h))
\\&=s_\lambda p_\mu\theta(\Phi^*(f))s_\mu \theta(\Phi^*(h))
\\&=s_\lambda s_\mu\ad s_\mu^*\circ\theta(\Phi^*(f)) \theta(\Phi^*(h))
\\&=s_{\lambda\mu}\theta\circ \tau_\mu^*(\Phi^*(f)) \theta(\Phi^*(h))
\\&=s_{\lambda\mu}\theta(\tau_\mu^*\circ\Phi^*(f)) \theta(\Phi^*(h))
\\&=s_{\lambda\mu}\theta(\Phi^*\circ\sigma_\mu^*(f)) \theta(\Phi^*(h))
\\&=s_{\lambda\mu}\theta\bigl(\Phi^*(\sigma_\mu^*(f)h)\bigr)
\\&=\rho_{\lambda\mu}(\sigma_\mu^*(f)h)
\\&=\rho_{\lambda\mu}(fh).
\end{align*}
It follows that $\rho:X\to C^*(\Lambda)$ is a representation.

Next we show that $\rho$ is Cuntz-Pimsner covariant.
Let $u\in\Lambda^0$, $n\in\N^k$, and $f\in X_u=C_0(T_u)$.
We need to show that
\[
\rho_u(f)=\sum_{\lambda\in u\Lambda^n}\rho^{(\lambda)}\circ
\varphi_\lambda(f),
\]
where
\[
\varphi_\lambda:C_0(T_u)\to \KK(X_\Lambda)
\]
is the left-module structure map.
We need a little more information regarding the homomorphism
\[
\rho^{(\lambda)}=\rho_\lambda^{(1)}:\KK(X_\lambda)\to C^*(\Lambda).
\]
For $\lambda\in u\Lambda v$ we have $X_\lambda={}_{\sigma_\lambda^*}C_0(T_v)$, so
\[
\KK(X_\lambda)=C_0(T_v),
\]
and for $g,h\in C_0(T_v)$ the rank-one operator $\theta_{g,h}$ is given by (left) multiplication by $g\bar h$. Thus
\begin{align*}
\rho^{(\lambda)}(g\bar h)
&=\rho_\lambda(g)\rho_\lambda(h)^*
\\&=s_\lambda\theta\circ\Phi^*(g)\theta\circ\Phi^*(\bar h)s_\lambda^*
\\&=s_\lambda\circ\theta\circ\Phi^*(g\bar h)s_\lambda^*
\\&=\ad s_\lambda\circ \rho_v(g\bar h).
\end{align*}
Since every function in $C_0(T_v)$ can be factored as $g\bar h$, we conclude that the homomorphism $\rho^{(\lambda)}$ coincides with
\[
\ad s_\lambda\circ \rho_v:C_0(T_v)\to C^*(\Lambda).
\]
Also, $\varphi_\lambda:C_0(T_u)\to \KK(X_\Lambda)$ coincides with $\sigma_\lambda^*:C_0(T_u)\to C_0(T_v)$ (note that $\sigma_\lambda^*$ maps into $C_0(T_v)$ because $\sigma_\lambda$ is proper).
Thus
\begin{align*}
\sum_{\lambda\in u\Lambda^n}\rho^{(\lambda)}\circ\varphi_\lambda(f)
&=\sum_{\lambda\in u\Lambda^n}\ad s_\lambda\circ\theta\circ\Phi^*\circ\sigma_\lambda^*(f)
\\&=\sum_{\lambda\in u\Lambda^n}\ad s_\lambda\circ\theta\circ\tau_\lambda^*\circ\Phi^*(f)
\\&=\sum_{\lambda\in u\Lambda^n}\ad s_\lambda\circ\ad s_\lambda^*\circ\theta\circ\Phi^*(f)
\\&=\sum_{\lambda\in u\Lambda^n}\ad s_\lambda s_\lambda^*\circ\theta\circ\Phi^*(f)
\\&=\sum_{\lambda\in u\Lambda^n}p_\lambda\theta\circ\Phi^*(f)
\\&=p_u\rho_u(f)
\righttext{}\Bigl(\text{since $\sum_{\lambda\in u\Lambda^n}p_\lambda=p_u$}\Bigr)
\\&=\rho_u(f),
\end{align*}
since $\rho_u(C_0(T_u))\subset D_u$ and $p_u=1_{D_u}$.

Therefore $\rho$ gives rise to a homomorphism $\Psi_\rho:\OO_X\to C^*(\Lambda)$ such that
\[
\Psi_\rho\circ \rho^X=\rho,
\]
where $\rho^X\to \OO_X$ is the universal Cuntz-Pimsner covariant representation.
For each $v\in\Lambda^0$, the continuous map $\Phi:\Lambda^\infty\to T$ takes $v\Lambda^\infty$ into $T_v$,
so $\Phi^*$ restricts to a nondegenerate homomorphism from $C_0(T_v)$ to $C(v\Lambda^\infty)$, and hence the homomorphism $\rho_v:C_0(T_v)\to D_v$ is nondegenerate.
It follows that for each $\lambda\in \Lambda v$ the generator $s_\lambda$ is in the range of $\rho_\lambda:X_\lambda\to C^*(\Lambda)$.
Thus $\Psi_\rho:\OO_X\to C^*(\Lambda)$ is surjective.

Finally, we appeal to the Gauge-Invariant Uniqueness Theorem \cite[Theorem~3.3.1]{system} to show that $\Psi_\rho$ is injective. Note that \cite{system} assume that the $\Lambda$-system $X$ is regular, while we only assume that it is $k$-regular; as we have mentioned before, $k$-regularity is all that's required to make the results of \cite{system} true.
First of all, for each $v\in\Lambda^0$, $\Phi$ maps $v\Lambda^0$ onto $T_v$,
and it follows that $\rho_v:C_0(T_v)\to D_u$ is faithful.
Thus the direct sum
\[
\Psi_\rho|_A=\bigoplus_{v\in\Lambda^0}\rho_v:\bigoplus_{v\in\Lambda^0}C_0(T_v)\to \bigoplus_{v\in\Lambda^0}D_v\subset C^*(\Lambda)
\]
is also faithful.
Let $\gamma:\T^k\to \aut C^*(\Lambda)$ be the gauge action.
For $\lambda\in \Lambda^nv$, $f\in C_0(T_v)$, and $z\in\T^k$,
\begin{align*}
\gamma_z\circ\rho_\lambda(f)
&=\gamma_z\bigl(s_\lambda\theta\circ\Phi^*(f)\bigr)
\\&=\gamma_z(s_\lambda)\rho_v(f)\righttext{(since $\rho_v(f)\in D_v\subset C^*(\Lambda)^\gamma$)}
\\&=z^ns_\lambda\rho_v(f)
\\&=z^n\rho_\lambda(f),
\end{align*}
so by \cite[Theorem~3.3.1]{system} $\Psi_\rho$ is faithful.
\end{proof}

\section{Mauldin-Williams $k$-graphs}\label{MW}

We continue to let $\Lambda$ be a row-finite $k$-graph with no sources.

\begin{prop}\label{unique}
Let $(T,\sigma)$ be a $\Lambda$-system of maps such that each $T_v$ is a complete metric space and, for some $c<1$ and
every $\lambda\in\Lambda$,
\[
\delta_v(\sigma_\lambda(t),\sigma_\lambda(s))\le c^{|d(\lambda)|}\delta_v(t,s)
\righttext{for all}\lambda\in\Lambda,t,s\in T_{s(\lambda)},
\]
where $\delta_v$ is the metric on $T_v$, $d$ is the degree functor for the $k$-graph $\Lambda$, and $|n|=\sum_{i=1}^kn_i$ for $n=(n_1,\dots,n_k)\in\N^k$.
Then 
there exists
a unique $k$-surjective $\Lambda$-subsystem $(K,\psi)$ such that each $K_v$ is a bounded closed subset of $T_v$, and in fact each $K_v$ is compact.
\end{prop}

Note that to check the hypothesis it suffices to show that each of the generating maps $\sigma_\lambda$ for $\lambda\in \Lambda^{e_i}$ has Lipschitz constant at most $c$, where $e_1,\dots,e_k$ is the standard basis for $\N^k$.

\begin{proof}
Let
\[
\CC=\prod_{v\in\Lambda^0}\CC(T_v),
\]
where for $v\in\Lambda^0$ we let $\CC(T_v)$ denote the set of bounded closed subsets of $T_v$, which is complete under the Hausdorff metric.
Since $\Lambda^0$ is countable, $\CC$ is a complete metric space. 
For each $n\in\N^k$ define a map $\wilde\sigma^n:\CC\to\CC$ by
\[
\wilde\sigma^n(C)_v=\bigcup_{\lambda\in v\Lambda^n}\sigma_\lambda(C_{s(\lambda)}).
\]
As in \cite{mauldinwilliams}, $\wilde\sigma^n$ is a contraction, and so
has a unique fixed point in $\CC$.
We need to know that  the maps $\{\wilde\sigma^n\}_{n\in\N^k}$  all have the same fixed point, and
it suffices to show that they commute.
Let $n,m\in\N^k$. Then for all $C=(C_v)_{v\in\Lambda^0}\in\CC$ and $v\in\Lambda^0$ we have
\begin{align*}
\wilde\sigma^n\circ\wilde\sigma^m(C)_v
&=\wilde\sigma^n\bigl(\wilde\sigma^m(C)\bigr)_v
\\&=\bigcup_{\lambda\in v\Lambda^n}\sigma_\lambda\bigl(\wilde\sigma^m(C)_{s(\lambda)}\bigr)
\\&=\bigcup_{\lambda\in v\Lambda^n}\sigma_\lambda\left(\bigcup_{\mu\in s(\lambda)\Lambda^m}\sigma_\mu\bigl(C_{s(\mu)}\bigr)\right)
\\&=\bigcup_{\lambda\in v\Lambda^n}\bigcup_{\mu\in s(\lambda)\Lambda^m}
\sigma_\lambda\circ\sigma_\mu\bigl(C_{s(\mu)}\bigr)
\\&=\bigcup_{\lambda\mu\in v\Lambda^{n+m}}\sigma_{\lambda\mu}\bigl(C_{s(\lambda\mu)}\bigr)
\\&=\bigcup_{\alpha\in v\Lambda^{n+m}}\sigma_{\alpha}\bigl(C_{s(\alpha)}\bigr),
\end{align*}
which, by the factorization property of $\Lambda$, coincides with
\[
\bigcup_{\mu\in v\Lambda^m}\bigcup_{\lambda\in s(\mu)\Lambda^n}
\sigma_\mu\circ\sigma_\lambda\bigl(C_{s(\lambda)}\bigr)
=\wilde\sigma^m\wilde\sigma^n(C)_v.
\]
Letting $(K_v)_{v\in\Lambda^0}$ be the unique common fixed point of $\wilde\sigma$ on $\CC$, we see that, setting $K=\bigcup_{v\in\Lambda^0}K_v$ and $\psi=\sigma|_K$, the restriction $(K,\psi)$ of $(T,\sigma)$ is the unique $k$-surjective $\Lambda$-subsystem with bounded closed subsets $K_v$.

To see that in fact every $K_v$ is compact,
play the same game with $\CC(T_v)$ replaced by the set of compact subsets of $T_v$,
again getting a unique fixed point. But since the compact subsets are among the bounded closed subsets, the resulting $\Lambda$-subsystem must coincide with the one we found above, by uniqueness.
\end{proof}

\begin{defn}
A \emph{Mauldin-Williams $\Lambda$-system} is 
a $k$-surjective $\Lambda$-system of maps $(T,\sigma)$
such that
each $T_v$ is a compact metric space and,
for some $c<1$,
every $\sigma_\lambda:T_{s(\lambda)}\to T_{r(\lambda)}$ is a contraction with Lipschitz constant at most $c^{|d(\lambda)|}$.
\end{defn}

\begin{prop}
Every Mauldin-Williams $\Lambda$-system $(T,\sigma)$ is self-similar,
and if $X$ is the associated $\Lambda$-system of correspondences then $\OO_X\cong C^*(\Lambda)$.
\end{prop}

\begin{proof}
We adapt the technique of Ionescu \cite{ionescu}.
Let $x\in v\Lambda^\infty$,
so that $x:\Omega_k\to\Lambda$ is a $k$-graph morphism.
For each $n\in\N^k$ let $x(0,n)$ be the unique path $\lambda\in\Lambda^n$ such that $x=\lambda y$ for some (unique) $y\in s(\lambda)\Lambda^\infty$.
By definition of Mauldin-Williams $\Lambda$-system, the range of each $\sigma_{x(0,n)}$ has diameter at most $c^{|n|}$.
Thus by compactness there is a unique element $\Phi(x)\in T_v$ such that
\[
\bigcap_{n\in\N^k}\sigma_{x(0,n)}(T_{s(x(0,n))})=\{\Phi(x)\}.
\]
We get a map $\Phi:\Lambda^\infty\to T$, which is continuous because for each $x\in\Lambda^\infty$ the images under $\Phi$ of the neighborhoods $x(0,n)\Lambda^\infty$ of $x$ have diameters shrinking to 0.
By construction, it's obvious that
\[
\Phi(\lambda x)=\sigma_\lambda(\Phi(x))\righttext{for all}\lambda\in\Lambda,x\in s(\lambda)\Lambda^\infty,
\]
so $\Phi$ is intertwining.

We show that $\Phi$ is surjective.
Put $T'=\Phi(\Lambda^\infty)$.
By \propref{T' self}, $(T',\sigma|_{T'})$ is $k$-surjective with each $T'_v$ compact,
which implies that $T'=T$ by the uniqueness in \propref{unique}.

Finally, it now follows from \thmref{main} that $\OO_X\cong C^*(\Lambda)$.
\end{proof}

\begin{rem}
It would be completely routine at this point to adapt Ionescu's techniques to prove a higher-rank version his other ``no-go theorem'' \cite[Theorem~3.4]{ionescu}, namely that there are no ``noncommutative Mauldin-Williams $\Lambda$-systems'' of maps.
\end{rem}

\section{The associated topological $k$-graph}\label{topological}

Let $\Lambda$ be a row-finite $k$-graph with no sources,
and let $(T,\sigma)$ be a $k$-regular $\Lambda$-system of maps.
We do \emph{not} assume that $(T,\sigma)$ is self-similar unless otherwise noted.

Let $(T,\sigma)$ be a $\Lambda$-system of maps.
We want to define a topological $k$-graph $\Lambda*T$ as follows:
\begin{enumerate}
\item
$\Lambda*T=\{(\lambda,t)\in \Lambda\times T:t\in T_{s(\lambda)}\}$:

\item
$s(\lambda,t)=(s(\lambda),t)$
and
$r(\lambda,t)=(r(\lambda),\sigma_\lambda(t))$;

\item
if $s(\lambda)=r(\mu)$ and $t=\sigma_\mu(s)$, then
$(\lambda,t)(\mu,s)=(\lambda\mu,s)$;

\item
$d(\lambda,t)=d(\lambda)$.
\end{enumerate}
$\Lambda*T$ has the relative topology from $\Lambda\times T$,
and
is the disjoint union of the open subsets $\{\lambda\}\times T_{s(\lambda)}$,
each of 
which
is a homeomorphic copy of $T_{s(\lambda)}$.

\begin{prop}\label{top graph}
The above operations make $\Lambda*T$ into a topological $k$-graph.
\end{prop}

\begin{proof}
%DO WE NEED MORE DETAIL?
This is routine.
For instance, it's completely routine to check that $\Lambda*T$ is a small category and the map defined in (4) is a functor.
Let's check the unique factorization property:
Let $(\lambda,t)\in \Lambda*T$ and $m,n\in\N^k$ with $d(\lambda)=m+n$.
Then we can uniquely write $\lambda=\mu\nu$ with $d(\mu)=m$ and $d(\nu)=n$.
We have
\[
(\lambda,t)=(\mu,\sigma_\nu(t))(\nu,t),
\quad
d(\mu,\sigma_\nu(t))=m,
\midtext{and}
d(\nu,t)=n,
\]
and $(\mu,\sigma_\nu(t))$ and $(\nu,t)$ are unique since $\mu$ and $\nu$ are.
It's immediate that the degrees match up and this factorization is unique.

The multiplication on the category $\Lambda*T$ is continuous 
and open because it is in fact a local homeomorphism
from the fibred product $(\Lambda*T)*(\Lambda*T)$ to $\Lambda*T$,
which for each $(\lambda,\mu)\in \Lambda\times \Lambda$ with $s(\lambda)=r(\mu)$
maps the open subset
\[
\Bigl(\bigl(\{\lambda\}\times T_{s(\lambda)}\bigr)
\times
\bigl(\{\mu\}\times T_{s(\mu)}\bigr)\Bigr)
\cap
\Bigl((\Lambda*T)*(\Lambda*T)\Bigr)
\]
bijectively onto the open subset $\{\lambda\mu\}\times T_{s(\mu)}$.

To see that the source map on $\Lambda*T$ is a local homeomorphism,
just note that it restricts to homeomorphisms
\[
\{\lambda\}\times T_{s(\lambda)}\to \{s(\lambda)\}\times T_{s(\lambda)}.
\qedhere
\]
\end{proof}

\begin{rem}
One could reasonably regard a $\Lambda$-system of maps as an action of $\Lambda$ on the space $T=\bigsqcup_{v\in \Lambda^0}T_v$,
and the topological $k$-graph $\Lambda*T$ as the associated transformation $k$-graph.
\end{rem}

\begin{rem}
If each $T_v$ is discrete and every map $\sigma_\lambda:T_{s(\lambda)}\to T_{r(\lambda)}$ is bijective, then the above $k$-graph $\Lambda*T$ coincides with that of \cite[Proposition~3.3]{pqr:cover}, where the main point was that the coordinate projection $(\lambda,t)\mapsto \lambda$ is a model for coverings of the $k$-graph $\Lambda$.
\end{rem}

\begin{prop}
Let $(T,\sigma)$ be a $k$-regular $\Lambda$-system of maps,
and let $(\AA,\varphi)$ be the associated $\Lambda$-system of homomorphisms,
which in turn has an associated $\Lambda$-system $X$ of correspondences.
Then
\[
\OO_X\cong C^*(\Lambda*T),
\]
where $\Lambda*T$ is the topological $k$-graph 
of
\propref{top graph}.
\end{prop}

\begin{proof}
Our strategy is to show that $\OO_X$ and $C^*(\Lambda*T)$ are isomorphic to the Cuntz-Pimsner algebras of isomorphic $\N^k$-systems of correspondences.
Recall that $\OO_X\cong \OO_Y$,
where $Y=Y_X$ is the $\N^k$ system associated to $X$.
Thus for each $n\in\N^k$ we have
\[
Y_n=\bigoplus_{\lambda\in \Lambda^n}X_\lambda,
\]
where $X_\lambda$ is the 
correspondence over $A=\bigoplus_{v\in\Lambda^0}A_v$
naturally associated
(via identifying the $A_v$'s with direct summands in $A$)
to the standard $A_{r(\lambda)}-A_{s(\lambda)}$ correspondence 
${}_{\sigma_\lambda^*}A_{s(\lambda)}$
determined by the homomorphism $\sigma_\lambda^*:A_{r(\lambda)}\to M(A_{s(\lambda)})$
given by composition with $\sigma_v:T_{s(v)}\to T_{r(v)}$.

On the other hand,
by \cite[Theorem~5.20]{carlsen} $C^*(\Lambda*T)$ is isomorphic to the Cuntz-Nica-Pimsner algebra $\NN\OO_Z$,
where $Z$ is the $\N^k$-system of 
$C_0((\Lambda*T)^0)$-correspondences associated to the topological $k$-graph $\Lambda*T$.
As we'll show in this proof, the $N^k$-systems $Z$ and $Y$ are isomorphic.
Since the $\Lambda$-system $(T,\sigma)$ is $k$-regular, so is $Y$, and hence so is $Z$.
In particular, since each pair in $\N^k$ has an upper bound,
and $C_0((\Lambda*T)^0)$ maps 
injectively 
into the compacts on $Z_n$ for every $n\in\N^k$,
it follows from \cite[Corollary~5.2]{simsyeend} that $\NN\OO_Z=\OO_Z$,
because
by \cite[Proposition~5.8]{fowlerdiscrete} $Z$ is compactly aligned.

Let's see what the $\Lambda$-system $Z$ looks like in this situation:
for each $n\in\N^k$,
the correspondence $Z_n$
over $C_0((\Lambda*T)^0)$
is a completion of $C_c((\Lambda*T)^n)$.
We can safely identify $(\Lambda*T)^0$ with $T=\bigsqcup_{v\in \Lambda^0}T_v$,
and hence $C_0((\Lambda*T)^0)$ with $A=\bigoplus_{v\in\Lambda^0}C_0(T_v)$,
and in this way $Z_n$ becomes an $A$-correspondence.
For $\xi,\eta\in C_c((\Lambda*T)^n)=C_c(\Lambda^n*T)$, the inner product is given by
\[
\<\xi,\eta\>_A(t)=\sum_{\lambda\in \Lambda^nv}\bar{\xi(\lambda,t)}\eta(\lambda,t),\quad t\in T_v,v\in\Lambda^0,
\]
and the right and left module operations are given for $f\in A$ by
\begin{align*}
&(\xi\cdot f)(\lambda,t)=\xi(\lambda,t)f(t)\\
&(f\cdot \xi)(\lambda,t)=f(\sigma_\lambda(t))\xi(\lambda,t).
\end{align*}

Note that
\[
(\Lambda*T)^n=\bigsqcup_{\lambda\in \Lambda^n}(\{\lambda\}\times T_{s(\lambda)}).
\]
Thus for each $\lambda\in\Lambda^nv$ we have a natural inclusion map
\[
C_c(\{\lambda\}\times T_v)\hookrightarrow Z_n,
\]
and $Z_n$ is the closed span of these subspaces.
Moreover, their closures form a pairwise orthogonal family of subcorrespondences of $Z_n$:
\[
Z_n(\lambda)=\bar{C_c(\{\lambda\}\times T_v)}\righttext{for}\lambda\in\Lambda^nv,
\]
and we see that
\[
Z_n=\bigoplus_{\lambda\in\Lambda^n}Z_n(\lambda)
\]
as $A$-correspondences.

We will obtain an isomorphism $\psi:Y\to Z$ of $\N^k$-systems by defining isomorphisms $\psi_n:Y_n\to Z_n$ of $A$-correspondences and
then verifying that
\[
\psi_n(\xi)\psi_m(\eta)=\psi_{n+m}(\xi\eta)\righttext{for all}(\xi,\eta)\in Y_n\times Y_m.
\]
By the above, to get an isomorphism $\psi_n:Y_n\to Z_n$ it suffices to get isomorphisms $\psi_{n,\lambda}:X_\lambda\to Z_n(\lambda)$ for each $\lambda\in\Lambda^n$.
If $\lambda\in \Lambda^nv$ and 
\[
\xi\in C_c(T_v)\subset X_\lambda
\]
define
\[
\psi(\xi)\in C_c(\{\lambda\}\times T_v)
\subset Z_n(\lambda)
\]
by
\[
\psi(\xi)(\lambda,t)=\xi(t).
\]
Routine computations show that $\psi_{n,\lambda}$ is an isomorphism.

Now we check multiplicativity, and again it suffices to consider the fibres of the $\Lambda$-system $X$: 
if
\begin{align*}
&\xi\in X_\lambda\righttext{for}\lambda\in \Lambda^nv\\
&\eta\in X_\mu\righttext{for}\mu\in v\Lambda^m
\end{align*}
then for $t\in T_{s(\mu)}$ we have
\begin{align*}
\bigl(\psi_{n,\lambda}(\xi)\psi_{m,\mu}(\eta)\bigr)(\lambda\mu,t)
&=\psi_{n,\lambda}(\xi)(\lambda,\sigma_\mu(t))\psi_{m,\mu}(\mu,t)
\\&=\xi(\sigma_\mu(t))\eta(t)
\\&=(\xi\eta)(t)
\\&=\bigl(\psi_{n+m,\lambda\mu}(\xi\eta)\bigr)(\lambda\mu,t).
\qedhere
\end{align*}
\end{proof}

\section{The tensor groupoids}\label{tensor groupoid}

Recall that in \cite{fowlersims} Fowler and Sims study what they call \emph{product systems taking values in a tensor groupoid}.
Their product systems are over semigroups, and here we want to consider the 
special cases related to our $\Lambda$-systems of homomorphisms or maps, where the $k$-graph $\Lambda$ has a single vertex, 
and so in particular is a 
monoid whose identity element is the unique vertex.
Since we won't need to do serious work with the concept, here we informally regard a \emph{tensor groupoid} as a groupoid $\GG$ with a ``tensor'' operation $X\otimes Y$ and an ``identity'' object $1_\GG$ such that the ``expected'' redistributions of parentheses and canceling of tensoring with the identity are implemented via given natural equivalences.
As defined in \cite{fowlersims}, 
a \emph{product system} 
over a semigroup $S$
taking values in a tensor groupoid $\GG$ 
is a family $\{X_s\}_{s\in S}$ of objects in $\GG$
together with an associative family $\{\alpha_{s,t}\}_{s,t\in S}$ of isomorphisms
\[
\alpha_{s,t}:X_s\otimes X_t\to X_{st},
\]
and moreover if $S$ has an identity $e$ then $X_e=1_\GG$ and $\alpha_{e,s},\alpha_{s,e}$ are the given isomorphisms $1_\GG\otimes X_s\cong X_s$ and $X_s\otimes 1_\GG\cong X_s$.

\subsection*{Systems of homomorphisms}

Let $A$ be a $C^*$-algebra, and  $\GG$ be the tensor groupoid whose objects are 
the nondegenerate homomorphisms $\pi:A\to M(A)$,
whose only morphisms are the identity morphisms on objects,
and with identity $1_\GG=\id_A$.
Define a tensor operation on $\GG$ by composition:
\[
\pi_1\otimes \pi_2=\pi_2\circ \pi_1,
\]
where $\pi_2$ has been canonically extended to a strictly continuous unital endomorphism of $M(A)$.
Standard properties of composition show that $\GG$ is indeed a tensor groupoid,
in a trivial way:
the tensor operation is actually associative, and $1_\GG$ acts as an actual identity for tensoring, 
so the axioms of \cite{fowlersims} for a tensor groupoid are obviously satisfied.

Due to the special nature of this tensor groupoid $\GG$,
a \emph{product system over $\N^k$ taking values in $\GG$},
as in \cite[Definition~1.1]{fowlersims},
is a homomorphism
$n\mapsto \varphi_n$
from the additive monoid $\N^k$ into the monoid of nondegenerate homomorphisms $A\to M(A)$ under composition,
in other words such a product system is precisely what we call in the current paper an $\N^k$-system of homomorphisms.

\subsection*{Systems of maps}

Quite similarly to the above,
let $T$ be a locally compact Hausdorff space, and  $\GG$ be the tensor groupoid whose objects are 
the continuous maps $\sigma:X\to X$,
whose only morphisms are the identity morphisms on objects,
and with identity $1_\GG=\id_X$.
Define a tensor operation on $\GG$ by composition:
\[
\sigma\otimes \psi=\sigma\circ\psi.
\]
Again, $\GG$ is indeed a tensor groupoid,
in a trivial way, because
the tensor operation is actually associative, and $1_\GG$ acts as an actual identity for tensoring.

A \emph{product system over $\N^k$ taking values in $\GG$},
as in \cite[Definition~1.1]{fowlersims},
is a homomorphism
$n\mapsto \sigma_n$
from the additive monoid $\N^k$ into the monoid of continuous selfmaps of $X$ maps under composition,
in other words such a product system is precisely what we call in the current paper an $\N^k$-system of maps.

\section{Reversing the processes}\label{reverse}

In \remref{processes} we noted that every $\Lambda$-system of maps gives rise to a $\Lambda$-system of homomorphisms, and every $\Lambda$-system of homomorphisms gives rise to a $\Lambda$-system of correspondences.
In this section we will investigate the extent to which these two processes are reversible.

\begin{q}
When is 
a given $\Lambda$-system of correspondences
isomorphic to 
the one
associated to a $\Lambda$-system 
of homomorphisms?
\end{q}

Investigating this question requires us to examine balanced tensor products of standard correspondences.
First we observe without proof the following elementary fact.

\begin{lem}\label{compose}
Let $\varphi:A\to M(B)$ and $\psi:B\to M(C)$ be nondegenerate homomorphisms.
Then there is a unique $A-C$ correspondence isomorphism
\[
\theta:{}_\varphi B\otimes_B {}_\psi C\iso {}_{\psi\circ\varphi} C
\]
such that
\[
\theta(b\otimes c)=\psi(b)c\righttext{for}b\in B,c\in C.
\]
\end{lem}

We can analyze the question
of whether a given $\Lambda$-system $X$ of correspondences
is isomorphic to one coming from a $\Lambda$-system 
of homomorphisms
in several steps:

First of all, 
without loss of generality we can look for a $\Lambda$-system of homomorphisms
of the form $(\AA,\varphi)$.

Next, for each $\lambda\in u\Lambda v$ the $A_u-A_v$ correspondence $X_\lambda$ must be isomorphic to a standard one,
more precisely
there must exist
a linear bijection
\[
\theta_\lambda:X_\lambda\to A_v
\]
and
a nondegenerate homomorphism
\[
\varphi_\lambda:A_u\to M(A_v)
\]
such that
\begin{align}\label{theta}
&\theta_\lambda(\xi)^*\theta_\lambda(\eta)=\<\xi,\eta\>_{A_v}
\righttext{for all}\xi,\eta\in X_\lambda\\
&\theta_\lambda(a\cdot \xi\cdot b)=\varphi_\lambda(a)\theta_\lambda(\xi)b
\righttext{for all}a\in A_u,\xi\in X_\lambda,b\in A_v.
\end{align}

Moreover, whenever $\lambda\in u\Lambda v,\mu\in v\Lambda w$ we must have
\begin{align*}
{}_{\varphi_{\lambda\mu}}A_w
&=X_{\lambda\mu}
\\&\cong X_\lambda\otimes_{A_v} X_\mu
\\&={}_{\varphi_\lambda}A_v\otimes_{A_v} {}_{\varphi_\mu}A_w
\\&\cong {}_{\varphi_\mu\circ\varphi_\lambda}A_w,
\end{align*}
so there exists a unitary multiplier $U(\lambda,\mu)\in M(A_w)$ such that
\[
\varphi_\mu\circ\varphi_\lambda=\ad U(\lambda,\mu)\circ\varphi_{\lambda\mu}.
\]
The $U(\lambda,\mu)$'s satisfy a kind of ``two-cocycle'' identity coming from associativity of composition of the $\varphi_\lambda$'s.

Now, if this $\Lambda$-system of correspondences is isomorphic to one associated to a $\Lambda$-system $(\AA,\psi)$ of homomorphisms,
then for each $\lambda\in u\Lambda v$ 
we must have an isomorphism
${}_{\varphi_\lambda}A_v\cong {}_{\psi_\lambda}A_v$
of $A_u-A_v$ correspondences,
and so there must be a unitary multiplier
$W_\lambda\in M(A_v)$
such that
\[
\varphi_\lambda=\ad W_\lambda\circ \psi_\lambda.
\]
Since $(\AA,\psi)$ is a $\Lambda$-system of homomorphisms,
whenever $\lambda\in u\Lambda v,\mu\in v\Lambda w$ we have
\begin{align*}
\varphi_{\lambda\mu}
&=\ad W_{\lambda\mu}\circ \psi_{\lambda\mu}
\\&=\ad W_{\lambda\mu}\circ \psi_\mu\circ\psi_\lambda
\\&=\ad W_{\lambda\mu}\circ \ad W_\mu^*\circ\varphi_\mu\circ \ad W_\lambda^*\circ\varphi_\lambda
\\&=\ad W_{\lambda\mu}W_\mu^*\varphi_\mu(W_\lambda^*)\circ\varphi_\mu\circ\varphi_\lambda
\\&=\ad W_{\lambda\mu}W_\mu^*\varphi_\mu(W_\lambda^*)U(\lambda,\mu)\circ\varphi_{\lambda\mu},
\end{align*}
so since the homomorphisms $\varphi_\lambda$ are nondegenerate
we see that, in the 
quotient group
of the unitary multipliers of $A_w$ modulo the 
central unitary multipliers,
the cosets satisfy
\[
\bigl[U(\lambda,\mu)\bigr]
=\bigl[\varphi_\mu(W_\lambda)W_\mu W_{\lambda\mu}^*\bigr],
\]
giving a sort of cohomological obstruction (which we won't make precise)
to the $\Lambda$-system of correspondences being isomorphic to a one associated to a $\Lambda$-system $(\AA,\psi)$ of homomorphisms.

Note that if all the $C^*$-algebras $A_v$ are commutative,
then none of the above unitary multipliers appear,
so once we have $\theta_\lambda$'s and $\varphi_\lambda$'s satisfying \eqref{theta}
then the pair $(\AA,\varphi)$ will automatically be a $\Lambda$-system of homomorphisms whose associated $\Lambda$-system of correspondences is isomorphic to $X$.
What makes this happen is the way in which the correspondences $X_\lambda$ fit together.
This is worth recording:

\begin{prop}\label{commutative}
Let $X$ be a $\Lambda$-system of correspondences
such that every $A_v$ is commutative.
Then $X$
is isomorphic to the $\Lambda$-system associated to a $\Lambda$-system of homomorphisms if and only if,
whenever $\lambda\in u\Lambda v$,
$X_\lambda$ is isomorphic to a standard $A_u-A_v$ correspondence ${}_{\varphi_\lambda}A_v$.
\end{prop}

\begin{prop}
Let $(\AA,\varphi)$ be a $\Lambda$-system of homomorphisms such that every $A_v$ is commutative,
and for each $v\in\Lambda^0$ let $T_v$ be the maximal ideal space of $A_v$.
Then there is a unique $\Lambda$-system of maps $(T,\sigma)$ such that $(\AA,\varphi)$ is the associated $\Lambda$-system of homomorphisms.
\end{prop}

On the other hand, every $\Lambda$-system of homomorphisms is uniquely isomorphic to the one associated to a $\Lambda$-system of maps, at least in the only circumstances where it makes sense:

\begin{proof}
This follows immediately from the duality between the category of commutative $C^*$-algebras and nondegenerate homomorphisms into multiplier algebras and the category of locally compact Hausdorff spaces and continuous maps.
\end{proof}

\section{No higher-rank fractals}\label{no go}

In \propref{T' self}  we showed that every $\Lambda$ system of maps
$(T,\sigma)$ has a self-similar $k$-surjective $\Lambda$-subsystem 
$(T',\sigma|_{T'})$. The self-similar set $T'$ is the part of the system
that would generally be referred to as the ``fractal''. It is natural to
wonder whether the generalization to $k$-graphs presented here gives rise
to any new fractals that could not have arisen from the corresponding 
constructions for $1$-graphs.  The answer to this question turns out to
be
``no'' for reasons we will now explain.
Throughout the following discussion, let $p=(1,1,\dots,1)\in \mathbb{N}^k$

\begin{defn}
For a $k$-graph $\Lambda$ we define the 
\textit{diagonal 1-graph} $E$ 
as follows:
\begin{align*}
E^0&=\Lambda^0 
\\
E^1&=\{e_{\lambda}:\lambda \in \Lambda, d(\lambda)=p\} 
\\
r(e_{\lambda})&=r(\lambda)
\\
s(e_{\lambda})&=s(\lambda).
\end{align*}
 If $(T,\sigma)$ is a $\Lambda$-system
of maps, then we define the \textit{diagonal $E$-system} $(T,\rho)$ of
$(T,\sigma)$ to be the $E$-system of maps such that 
$\rho_{e_{\lambda}}=\sigma_{\lambda}$
for all $e_{\lambda}\in E^1$.
Finally, let $\alpha:E^*\to\Lambda$ be the map defined by $\alpha(e_{\lambda_1}e_{\lambda_2}\cdots
e_{\lambda_n})=\lambda_1\lambda_2\cdots\lambda_n$.
\end{defn}

\begin{prop}
The map $i:\Lambda^{\infty}\to E^{\infty}$ defined
by $\alpha (i(x)(j,l))=x(jp,lp)$ is a bijection
and $i^{-1}$ is continuous.
\end{prop}

\begin{proof}
First we must show that this is well-defined. This just amounts
to showing that $\alpha$ is injective. To see this recall that 
if $\alpha(e_{\lambda_1}e_{\lambda_2}\cdots e_{\lambda_n})=\lambda$ then
$\lambda = \lambda_1\lambda_2\cdots\lambda_n$ where each $\lambda_i$
has degree $p$ and hence $d(\lambda_1\lambda_2\cdots\lambda_n)=np$. Since
there is only one way to write $np$ as a sum of $p$'s, there is only
one such decomposition of $\lambda$ (by unique factorization),
so if $\alpha(e_{\lambda_1}e_{\lambda_2}\cdots
e_{\lambda_n})=\alpha(e_{\gamma_1}e_{\gamma_2}\cdots e_{\gamma_n})$ we must have
$\lambda_i=\gamma_i$ for all $i$.

Next, to show that $i$ is injective, suppose $i(x)=i(y)$ for $x,y\in
\Lambda^{\infty}$. Then by definition we must have that $x(jp,lp)=y(jp,lp)$
for all $j,l\in\mathbb{N}$, and in particular we have that $x(0,jp)=y(0,jp)$
for all $j\in\mathbb{N}$. But since $\{jp\}_j$ is a cofinal increasing sequence
in $\mathbb{N}^k$, $x$ and $y$ are uniquely determined by their values
on the pairs $(0,jp)$ (see \cite[Remarks 2.2]{kp}) so we must have $x=y$.

Now, to show that $i$ is surjective, let $z\in E^{\infty}$. We wish
to find an infinite path $x\in\Lambda^{\infty}$ such that $i(x)=z$. 
We will again make use of the fact that such an $x$ is uniquely determined by its values on $(0,jp)$. Specifically, if $z(0,j)=e_{\lambda_1}e_{\lambda_2}
\cdots e_{\lambda_j}$, then we let $x(0,jp)=\lambda_1\lambda_2\cdots\lambda_j$.
Then $\alpha(i(x)(0,j))=x(0,jp)=\lambda_1\lambda_2\cdots\lambda_j$ and 
$\alpha(z(0,j))=\lambda_1\lambda_2\cdots\lambda_j$ so by the injectivity of
$\alpha$ we have that $i(x)(0,j)=z(0,j)$ and since $i(x)$ and $z$ are
uniquely determined by their values at $(0,j)$ we have that $i(x)=z$.

Finally, we need to show that $i^{-1}$ is continuous. We have
\[
\alpha (i(x)(j,l))=x(jp,lp)=\lambda_j...\lambda_l,
\]
where 
$\lambda_j...\lambda_l$ is the unique decomposition of $x(jp,lp)$ into
paths of degree $p$. Since $\alpha$ is injective, we get $i(x)(j,l)
=e_{\lambda_j}...e_{\lambda_l}$. Since this holds for all $(j,l)$ we
must have that $i^{-1}(e_{\lambda_1}e_{\lambda_2}...)=\lambda_1
\lambda_2...$ for all $e_{\lambda_1}e_{\lambda_2}...\in E^{\infty}$.
Recall that the topologies on 
$E^{\infty}$ and $\Lambda^{\infty}$ are generated by the collections 
$\{Z(P):P\in E^*\}$ and $\{Z(\lambda):\lambda \in \Lambda\}$ respectively
where 
$Z(P)=\{Pz:z\in s(P)E^{\infty}\}$
and 
$Z(\lambda)=\{\lambda x:x\in s(\lambda)\Lambda^{\infty}\}$.
Thus a net $\{\lambda^{\alpha}_1\lambda
^{\alpha}_2...\}_{\alpha\in A}$ in $\Lambda^{\infty}$
converges to $\lambda_1\lambda_2...$ in $\Lambda^{\infty}$ if for all 
$n\in \mathbb{N}$ there is
$\alpha_0\in A$ such that $\lambda^{\alpha}_j=\lambda_j$ 
for all $j\leq n$
and $\alpha\ge \alpha_0$,
and similarly
for nets in $E^{\infty}$. Now, suppose $\{e_{\lambda^{\alpha}_1}
e_{\lambda^{\alpha}_2}...\}_{\alpha\in A}$ 
converges to $e_{\lambda_1}e_{\lambda_2}...$ in $E^{\infty}$. Then for all 
$n\in \mathbb{N}$ there is
$\alpha_0\in A$ such that $e_{\lambda^{\alpha}_j}=e_{\lambda_j}$ 
for all $j\leq n$
and $\alpha\ge \alpha_0$.
Thus
$\lambda^{\alpha}_j=\lambda_j$ 
for all $j\leq n$ and $\alpha\ge \alpha_0$,
and we have shown that
the net 
$\{i^{-1}(e_{\lambda^{\alpha}_1}e_{\lambda^{\alpha}_2}...)\}_{\alpha\in A}=
\{\lambda^{\alpha}_1\lambda
^{\alpha}_2...\}_{\alpha\in A}$ 
converges to $i^{-1}(e_{\lambda_1}e_{\lambda_2}...)=\lambda_1\lambda_2...$
in $\Lambda^{\infty}$.
Therefore $i^{-1}$ is continuous.
\end{proof}

\begin{prop}\label{b}
Let $(T,\sigma)$ be a $\Lambda$-system of maps and
let $(T,\rho)$ be the diagonal $E$-system of $(T,\sigma)$.
If $\Phi:\Lambda^{\infty} \to T$ is intertwining with respect to $(T,\sigma)$
then $\Phi \circ i^{-1}:E^{\infty}\to T$ is intertwining with respect to
$(T,\rho)$.
\end{prop}

\begin{proof}
We have:
\begin{displaymath}
\Phi \circ i^{-1}\circ \tau_{e_{\lambda}}(x)=\Phi(i^{-1}(e_{\lambda}x))
=\Phi(\lambda i^{-1}(x))=\Phi \circ \tau_{\lambda}(i^{-1}(x))
\end{displaymath} 
but since $\Phi$ is intertwining, this gives:
\begin{displaymath}
=\sigma_{\lambda} \circ \Phi(i^{-1}(x))=\rho_{e_{\lambda}}\circ \Phi \circ
i^{-1}(x).
\end{displaymath}
Since $x$ was arbitrary, we have $(\Phi\circ i^{-1}) \circ \tau_{\lambda}=
\rho_{\lambda} \circ (\Phi \circ i^{-1})$ so $\Phi \circ i^{-1}$ is 
intertwining with respect to $(T,\rho)$.
\end{proof}

\begin{defn}
If $(T,\sigma)$ is a $\Lambda$
system of maps, $\Phi$ is an intertwining map, and $(T',\sigma|_{T'})$
is the self-similar $k$-surjective
$\Lambda$-subsystem of 
\propref{T' self}, then we call
 $T'$ the \textit{attractor} of $(T,\sigma,\Phi)$.
\end{defn}

\begin{thm}Let $\Lambda$ be a $k$-graph. Suppose $(T,\sigma)$ is
a $\Lambda$-system 
of maps, $\Phi$ is an intertwining map with respect to $(T,\sigma)$, and
$T'$ is the attractor of $(T,\sigma,\Phi)$. Then there
exist a $1$-graph $E$ with $E^0=\Lambda^0$, an $E$-system of
maps $(T,\rho)$, and an intertwining map $\Psi$ with respect to $(T,\rho)$
such that if $T''$ is the attractor of $(T,\rho,\Psi)$ then $T''=T'$.
\end{thm}

\begin{proof}
Let $E$ be the diagonal $1$-graph of $\Lambda$, $(T,\rho)$
be the diagonal $E$-system of $(T,\sigma)$ , and $\Psi=\Phi \circ i^{-1}$.
\propref{b} shows that this is  an intertwining map. For all $v\in \Lambda^0$ we have
\begin{displaymath}
T''_v=\Psi(vE^{\infty})=\Phi(i^{-1}(vE^{\infty}))=\Phi(v\Lambda^{\infty})=T'_{v},
\end{displaymath}
and hence $T''=T'$.
\end{proof}

%\bibliographystyle{amsalpha}
%\bibliography{cstar}

\providecommand{\bysame}{\leavevmode\hbox to3em{\hrulefill}\thinspace}
\providecommand{\MR}{\relax\ifhmode\unskip\space\fi MR }
\providecommand{\MRhref}[2]{%
  \href{http://www.ams.org/mathscinet-getitem?mr=#1}{#2}
}
\providecommand{\href}[2]{#2}

\end{document}